\documentclass[a4paper,12pt]{amsart}
\title[Rational function semifields are finitely generated]{Rational function semifields of tropical curves are finitely generated over the tropical semifield}

\author{Song JuAe}
\address{Tokyo Metropolitan University 1-1 Minami-Ohsawa, Hachioji, Tokyo, 192-0397, Japan.}
\email{song-juae@ed.tmu.ac.jp}


\usepackage{amsmath,amssymb,amscd}
\usepackage{amsthm}
\usepackage{color}
\usepackage{subfigure}
\usepackage{comment}
\usepackage{listings,jvlisting}
\usepackage{mathrsfs}
\usepackage{algorithm}
\usepackage{algorithmic}

\newtheorem{dfn}{Definition}[section]
\newtheorem{thm}[dfn]{Theorem}
\newtheorem{prop}[dfn]{Proposition}
\newtheorem{cor}[dfn]{Corollary}
\newtheorem{lemma}[dfn]{Lemma}
\newtheorem{rem}[dfn]{Remark}
\newtheorem{ex}[dfn]{Example}

\def\Gamma{\varGamma}

\begin{document}

\maketitle

\begin{abstract}
We prove that the rational function semifield of a tropical curve is finitely generated as a semifield over the tropical semifield $\boldsymbol{T} := ( \boldsymbol{R} \cup \{ - \infty \}, \operatorname{max}, +)$ by giving a specific finite generating set.
Also, we show that for a finite harmonic morphism between tropical curves $\varphi : \Gamma \to \Gamma^{\prime}$, the rational function semifield of $\Gamma$ is finitely generated as a $\varphi^{\ast}(\operatorname{Rat}(\Gamma^{\prime}))$-algebra, where $\varphi^{\ast}(\operatorname{Rat}(\Gamma^{\prime}))$ stands for the pull-back of the rational function semifield of $\Gamma^{\prime}$ by $\varphi$.
\end{abstract}

\section{Introduction}
	\label{introduction}

This paper gives a tropical analogue of the fact that the function field of an algebraic curve over $\boldsymbol{C}$ is generated by two elements over $\boldsymbol{C}$:

\begin{thm}
	\label{thm1}
Let $\Gamma$ be a tropical curve.
Then, the rational function semifield ${\rm Rat}(\Gamma)$ of $\Gamma$ is finitely generated as a semifield over the tropical semifield $\boldsymbol{T} := ( \boldsymbol{R} \cup \{ -\infty \}, \operatorname{max}, +)$.
\end{thm}

Here, a tropical curve is a metric graph that may have edges of length $\infty$, and a rational function on a tropical curve is a piecewise affine continuous function with integer slopes and with a finite number of pieces or a constant $-\infty$ function.
The set $\operatorname{Rat}(\Gamma)$ of all rational functions on a tropical curve $\Gamma$ has a natural  structure of a semifield over $\boldsymbol{T}$, where the addition $\oplus$ is defined as the pointwise maximum operation and the multiplication $\odot$ as the pointwise usual addition.

The following lemma is our key to prove Theorem \ref{thm1}:

\begin{lemma}[{\cite[Lemma 2.4.2]{JuAe}}]
	\label{lem:chipfiring}
Let $\Gamma$ be a tropical curve.
Then, $\operatorname{Rat}(\Gamma)$ is generated by all chip firing moves and all constant functions as a group with tropical multiplication $\odot$ as its binary operation.
\end{lemma}

Here, a chip firing move $\operatorname{CF}(\Gamma_1, l)$ is the rational function defined by the pair of a subgraph $\Gamma_1$ and a number $l \in \boldsymbol{R}_{>0} \cup \{ \infty \}$ as follows: $\operatorname{CF}(\Gamma_1, l) (x) := - \operatorname{min} \{ \operatorname{dist}(\Gamma_1, x), l \}$, where $\Gamma_1$ has no connected components consisting only of a point at infinity and $\operatorname{dist}(\Gamma_1, x)$ denotes the distance between $\Gamma_1$ and $x$.
By this lemma, it is enough to find a finite set of rational functions which generates all chip firing moves as a semifield over $\boldsymbol{T}$.

The following example suggests that all chip firing moves defined by one point generate all chip firing moves.
Here, the valence $\operatorname{val}(x)$ of a point $x$ of a tropical curve is the minimum number of the connceted components $U \setminus \{ x \}$ with all neighborhoods $U$ of $x$.

\begin{ex}
	\label{ex1}
\upshape{
Let $\Gamma$ be a tropical curve.
Let $\Gamma_1$ be a proper subgraph of $\Gamma$ that has no connected components consisting only of a point at infinity.
For any $l \in \boldsymbol{R} \cup \{ \infty \}$ and any $l^{\prime}$ such that $0 < l^{\prime} \le l$, we can \textit{cut} the bottom side of the chip firing move $\operatorname{CF}(\Gamma_1, l)$:
\[
\operatorname{CF}(\Gamma_1, l^{\prime}) = \operatorname{CF}(\Gamma_1, l) \oplus (-l^{\prime}).
\]

Let $\Gamma_2 := \{ x \in \Gamma \,|\, \operatorname{dist}(\Gamma_1, x) \le l^{\prime} \}$.
We can \textit{cut} the top side of the chip firing move $\operatorname{CF}(\Gamma_1, l)$:
\[
\operatorname{CF}(\Gamma_2, l - l^{\prime}) = \left\{ \operatorname{CF}(\Gamma_1, l)^{\odot (-1)} \oplus l^{\prime} \right\}^{\odot (-1)} \odot l^{\prime}.
\]

We can \textit{extend} the chip firing move $\operatorname{CF}(\Gamma_1, l^{\prime})$:
\[
\operatorname{CF}(\Gamma_1, l) = \operatorname{CF}(\Gamma_1, l^{\prime}) \odot \operatorname{CF}(\Gamma_2, l - l^{\prime}).
\]

Let $x$ be a boundary point of $\Gamma_1$ in $\Gamma$.
Let $\varepsilon$ be a sufficiently small positive real number and $y \in \Gamma \setminus \Gamma_1$ such that $\operatorname{dist}(x, y) = \varepsilon$.
Then, we can \textit{connect} two chip firing moves $\operatorname{CF}(\Gamma_1, \varepsilon)$ and $\operatorname{CF}(\{ y \}, \varepsilon)$:
\[
\operatorname{CF}(\Gamma_1 \cup [x, y], \varepsilon) = \operatorname{CF}(\Gamma_1, \varepsilon) \odot \operatorname{CF}(\{ y \}, \varepsilon) \odot \varepsilon.
\]

Let $\Gamma_3, \Gamma_4$ be any two proper subgraphs of $\Gamma$ whose intersection is empty and both that have no connected components consisting only of a point at infinity.
Let $l$ be a positive real number such that the intersection of $\{ x \in \Gamma \,|\, \operatorname{dist}(\Gamma_3, x) \le l \}$ and $\{ x \in \Gamma \,|\, \operatorname{dist}(\Gamma_4, x) \le l \}$ is finite.
Then, we have
\[
\operatorname{CF}(\Gamma_3 \sqcup \Gamma_4, l) = \operatorname{CF}(\Gamma_3, l) \oplus \operatorname{CF}(\Gamma_4, l).
\]
}
\end{ex}

Note that Algorithm \ref{alg:connection} in Section \ref{section:main results} gives a range of ``$\varepsilon$ is sufficiently small".
If $\Gamma$ has no edges of length $\infty$, then in fact all chip firing moves defined by one point generate all chip firing moves.
By the following lemma, we may assume that $\Gamma$ has no edges of length $\infty$.

\begin{lemma}
	\label{lem:troptometric}
Let $\Gamma$ be a tropical curve.
Let $\Gamma^{\prime}$ be a tropical curve which is obtained from $\Gamma$ by contracting edges of length $\infty$. 
If $\operatorname{Rat}(\Gamma^{\prime})$ is finitely generated as a semifield over $\boldsymbol{T}$, then so is $\operatorname{Rat}(\Gamma)$.
\end{lemma}

Hence our next target is to find a finite set of rational functions which generates all chip firing moves defined by one point as a semifield over $\boldsymbol{T}$.
Let $(G_{\circ}, l_{\circ})$ be the canonical model for $\Gamma$, i.e., the pair of the underlying graph $G_{\circ}$ of $\Gamma$ whose set $V(G_{\circ})$ of vertices is $\{ x \in \Gamma \,|\, \operatorname{val}(x) \not= 2 \}$ and the length function $l_{\circ}$ defined by $\Gamma$ and $G_{\circ}$ (for more precisely, see Subsection \ref{subsection2.2}).
Fix a direction on edges of $G_{\circ}$.
Let each $e \in E(G_{\circ})$ be identified with the interval $[0, l_{\circ}(e)]$ with this direction, where $E(G_{\circ})$ denotes the set of edges of $G_{\circ}$.
For each edge $e \in E(G_{\circ})$, let $x_e = \frac{l_{\circ}(e)}{4}$, $y_e = \frac{l_{\circ}(e)}{2}$, and $z_e = \frac{3 l_{\circ}(e)}{4}$.
We define rational functions 
\[
f_e := \operatorname{CF} \left( \{ y_e \}, \frac{l_{\circ}(e)}{2} \right),
g_e := \operatorname{CF} \left( \{ x_e \}, \frac{l_{\circ}(e)}{4} \right), 
h_e := \operatorname{CF} \left( \{ z_e \}, \frac{l_{\circ}(e)}{4} \right).
\]
Note that the semifield generated by $g_e, h_e$ over $\boldsymbol{T}$ coincides with the semifield generated by $g_e \odot h_e^{\odot (-1)}$ over $\boldsymbol{T}$ since
\[
g_e = \left( - \frac{l_{\circ}(e)}{4} \right) \odot \left( g_e \odot h_e^{\odot (-1)} \oplus 0 \right)
\]
and
\[
h_e = \left( - \frac{l_{\circ}(e)}{4} \right) \odot \left\{ \left( g_e \odot h_e^{\odot (-1)} \right)^{\odot (-1)} \oplus 0 \right\}.
\]
Let $R$ be the semifield generated by $f_e$, $g_e$, $h_e$ for any $e \in E(G_{\circ})$ and $\operatorname{CF}(\{ v \}, \infty) (= -\operatorname{dist}(v, \cdot))$ for any $v \in V(G_{\circ})$ over $\boldsymbol{T}$.
This semifield $R$ is finitely generated, and in fact, coincides with $\operatorname{Rat}(\Gamma)$.
Hence, we have Theorem \ref{thm1}.

In the setting that a finite harmonic morphism between tropical curves is given, we have the following proposition:

\begin{prop}
	\label{prop:morphismsemiring}
Let $\varphi : \Gamma \to \Gamma^{\prime}$ be a finite harmonic morphism between tropical curves.
Then, $\operatorname{Rat}(\Gamma)$ is finitely generated as a $\varphi^{\ast}(\operatorname{Rat}(\Gamma^{\prime}))$-algebra, where $\varphi^{\ast}(\operatorname{Rat}(\Gamma^{\prime}))$ stands for the pull-back of $\operatorname{Rat}(\Gamma^{\prime})$ by $\varphi$.
\end{prop}

Note that $\operatorname{Rat}(\Gamma)$ may not be finitely generated as a $\varphi^{\ast}(\operatorname{Rat}(\Gamma^{\prime}))$-module.
See Example \ref{ex2}.

This paper is organized as follows.
In Section $2$, we give basic definitions related to semirings and tropical curves which we need later.
Section $3$ gives proofs of Theorem \ref{thm1}, Lemma \ref{lem:troptometric}, and Proposition \ref{prop:morphismsemiring}.
In that section, we also show that there exists a generating set of rational function semifield of any tree whose elements are fewer than that of the above generating set and that rational function semifields of tropical curves other than a singleton are not finitely generated as a $\boldsymbol{T}$-algebra.

\section*{Acknowledgements}
The author thanks my supervisor Masanori Kobayashi, Yuki Kageyama, Yasuhito Nakajima, Ken Sumi, and Daichi Miura for helpful comments.
This work was supported by JSPS KAKENHI Grant Number 20J11910.

\section{Preliminaries}

In this section, we prepare basic definitions related to semirings and tropical curves which we need later.
For an introduction to the theory of tropical geometry, for example, see \cite{Maclagan=Sturmfels}.
We employ definitions in \cite{JuAe} for tropical curves.

\subsection{Semirings}
	\label{subsection2.1}
 
In this paper, a \textit{semiring} is a commutative semiring with the absorbing neutral element $0$ for addition and the identity $1$ for multiplication such that $0 \not= 1$.
If every nonzero element of a semiring $S$ is multiplicatively invertible, then $S$ is called a \textit{semifield}.

A map $\varphi : S_1 \to S_2$ between semirings is a \textit{semiring homomorphism} if for any $x, y \in S_1$,
\[
\varphi(x + y) = \varphi(x) + \varphi(y), \	\varphi(x \cdot y) = \varphi(x) \cdot \varphi(y), \	\varphi(0) = 0, \	\text{and}\	\varphi(1) = 1.
\]
Given a semiring homomorphism $\varphi : S_1 \to S_2$, we call the pair $(S_2, \varphi)$ (for short, $S_2$) a \textit{$S_1$-algebra}.

The set $\boldsymbol{T} := \boldsymbol{R} \cup \{ -\infty \}$ with two tropical operations:
\[
a \oplus b := \operatorname{max}\{ a, b \} \	\text{and} \	a \odot b := a + b,
\]
where both $a$ and $b$ are in $\boldsymbol{T}$, becomes a semifield.
Here, for any $a \in \boldsymbol{T}$, we handle $-\infty$ as follows:
\[
a \oplus (-\infty) = (-\infty) \oplus a = a \	\text{and}\	a \odot (-\infty) = (-\infty) \odot a = -\infty.
\]
$\boldsymbol{T}$ is called the \textit{tropical semifield}.

\subsection{Tropical curves}
	\label{subsection2.2}

In this paper, a \textit{graph} is an unweighted, undirected, finite, connected nonempty multigraph that may have loops.
For a graph $G$, the set of vertices is denoted by $V(G)$ and the set of edges by $E(G)$.
The \textit{valence} of a vertex $v$ of $G$ is the number of edges incident to $v$, where each loop is counted twice.
A vertex $v$ of $G$ is a \textit{leaf end} if $v$ has valence one.
A \textit{leaf edge} is an edge of $G$ incident to a leaf end.

An \textit{edge-weighted graph} $(G, l)$ is the pair of a graph $G$ and a function $l: E(G) \to {\boldsymbol{R}}_{>0} \cup \{\infty\}$, where $l$ can take the value $\infty$ on only leaf edges.
A \textit{tropical curve} is the underlying topological space of an edge-weighted graph $(G, l)$ together with an identification of each edge $e$ of $G$ with the closed interval $[0, l(e)]$.
The interval $[0, \infty]$ is the one point compactification of the interval $[0, \infty)$.
We regard $[0, \infty]$ not just as a topological space but as almost a metric space.
The distance between $\infty$ and any other point is infinite.
When $l(e)=\infty$, the leaf end of $e$ must be identified with $\infty$.
If $E(G) = \{ e \}$ and $l(e)=\infty$, then we can identify either leaf ends of $e$ with $\infty$.
When a tropical curve $\Gamma$ is obtained from $(G, l)$, the edge-weighted graph $(G, l)$ is called a \textit{model} for $\Gamma$.
There are many possible models for $\Gamma$.
A model $(G, l)$ is \textit{loopless} if $G$ is loopless.
We frequently identify a vertex (resp. an edge) of $G$ with the corresponding point (resp. the corresponding closed subset) of $\Gamma$.
For a point $x$ on a tropical curve $\Gamma$ obtained from $(G, l)$, if $x$ is identified with $\infty$, then $x$ is called a \textit{point at infinity}, else, $x$ is called a \textit{finite point}.
$\Gamma_{\infty}$ denotes the set of all points at infinity of $\Gamma$.
If $\Gamma_{\infty}$ is empty, i.e. $l : E(G) \to \boldsymbol{R}_{>0}$, then $\Gamma$ is called a \textit{metric graph}.
If $x$ is a finite point, then the \textit{valence} $\operatorname{val}(x)$ is the number of connected components of $U \setminus \{ x \}$ with any sufficiently small connected neighborhood $U$ of $x$, if $x$ is a point at infinity, then $\operatorname{val}(x) := 1$.
Remark that this ``valence" is defined for a point of a tropical curve and the ``valence" in the first paragraph of this subsection is defined for a vertex of a graph, and these are compatible with each other.
We construct a model $(G_{\circ}, l_{\circ})$ called the \textit{canonical model} for $\Gamma$ as follows.
We determine $V(G_{\circ}) := \{ x \in \Gamma \,|\, \operatorname{val}(x) \not= 2 \}$ except following two cases.
When $\Gamma$ is homeomorphic to a circle $S^1$, we determine $V(G_{\circ})$ as the set consisting of one arbitrary point of $\Gamma$.
When $\Gamma$ has the edge-weighted graph $(T, l)$ as its model, where $T$ is the tree consisting of three vertices and two edges and $l(E(T)) = \{ \infty \}$, we determine $V(G_{\circ})$ as the set of two points at infinity and any finite point of $\Gamma$.
The \textit{relative interior} $e^{\circ}$ of an edge $e$ is $e \setminus \{v, w\}$ with the endpoint(s) $v, w$ of $e$.
The \textit{genus} $g(\Gamma)$ of $\Gamma$ is the first Betti number of $\Gamma$, which coincides with $\# E(G) - \# V(G) + 1$ for any model $(G, l)$ for $\Gamma$.
A \textit{tree} is a tropical curve of genus zero.
The word ``an edge of $\Gamma$" means an edge of $G$ with some model $(G, l)$ for $\Gamma$.

\subsection{Rational functions and chip firing moves}
	\label{subsection2.3}

Let $\Gamma$ be a tropical curve.
A continuous map $f : \Gamma \to \boldsymbol{R} \cup \{ \pm \infty \}$ is a \textit{rational function} on $\Gamma$ if $f$ is a piecewise affine function with integer slopes, with a finite number of pieces and that can take the value $\pm \infty$ only at points at infinity, or a constant function of $-\infty$.
For a point $x$ of $\Gamma$ and a rational function $f \in \operatorname{Rat}(\Gamma) \setminus \{ -\infty \}$, $x$ is a \textit{pole} of $f$ if the sign of the sum of outgoing slopes of $f$ at $x$ is minus.
The absolute value of the sum is its degree.
Let $\operatorname{Rat}(\Gamma)$ denote the set of all rational functions on $\Gamma$.
For rational functions $f, g \in \operatorname{Rat}(\Gamma)$ and a point $x \in \Gamma \setminus \Gamma_{\infty}$, we define
\[
(f \oplus g) (x) := \operatorname{max}\{f(x), g(x)\}~~~\text{and}~~~(f \odot g) (x) := f(x) + g(x).
\]
We extend $f \oplus g$ and $f \odot g$ to points at infinity to be continuous on whole $\Gamma$.
Then both are rational functions on $\Gamma$.
Note that for any $f \in \operatorname{Rat}(\Gamma)$, $f \odot (-\infty) = (-\infty) \odot f = -\infty$.
Then $\operatorname{Rat}(\Gamma)$ becomes a semifield with these two operations.
Also, $\operatorname{Rat}(\Gamma)$ becomes a $\boldsymbol{T}$-algebra with the natural inclusion $\boldsymbol{T} \hookrightarrow \operatorname{Rat}(\Gamma)$.
Note that for $f, g \in \operatorname{Rat}(\Gamma)$, $f = g$ means that $f(x) = g(x)$ for any $x \in \Gamma$.

A \textit{subgraph} of a tropical curve is a compact nonempty subset of the tropical curve with a finite number of connected components.
Let $\Gamma_1$ be a subgraph of a tropical curve $\Gamma$ which does not have any connected components consisting of only points at infinity and $l$ a positive real number or infinity.
The \textit{chip firing move} by $\Gamma_1$ and $l$ is defined as the rational function $\operatorname{CF}(\Gamma_1, l)(x) := - \operatorname{min}(\operatorname{dist}(x, \Gamma_1), l)$.

\subsection{Finite harmonic morphisms}
	\label{subsection2.4}

Let $\varphi : \Gamma \to \Gamma^{\prime}$ be a continuous map between tropical curves.
$\varphi$ is a \textit{finite harmonic morphism} if there exist loopless models $(G, l)$ and $(G^{\prime}, l^{\prime})$ for $\Gamma$ and $\Gamma^{\prime}$, respectively, such that $(1)$ $\varphi(V(G)) \subset V(G^{\prime})$ holds, $(2)$ $\varphi(E(G)) \subset E(G^{\prime})$ holds, $(3)$ for any edge $e$ of $G$, there exists a positive integer $\operatorname{deg}_e(\varphi)$ such that for any points $x, y$ of $e$, $\operatorname{dist}(\varphi (x), \varphi (y)) = \operatorname{deg}_e(\varphi) \cdot \operatorname{dist}(x, y)$ holds, and $(4)$ for every vertex $v$ of $G$, the sum $\sum_{e \in E(G):\, e \mapsto e^{\prime},\, v \in e} \operatorname{deg}_e(\varphi)$ is independent of the choice of $e^{\prime} \in E(G^{\prime})$ incident to $\varphi(v)$.
This sum is denoted by $\operatorname{deg}_v(\varphi)$.
Then, the sum $\sum_{v \in V(G):\, v \mapsto v^{\prime}} \operatorname{deg}_v(\varphi)$ is independent of the choice of a vertex $v^{\prime}$ of $G^{\prime}$, and is called the \textit{degree} of $\varphi$.
If both $\Gamma$ and $\Gamma^{\prime}$ are singletons, we regard $\varphi$ as a finite harmonic morphism that can have any number as its degree.

Let $\varphi : \Gamma \to \Gamma^{\prime}$ be a finite harmonic morphism between tropical curves.
The \textit{pull-back} $\varphi^{\ast} : \operatorname{Rat}(\Gamma^{\prime}) \to \operatorname{Rat}(\Gamma)$ is defined by $f^{\prime} \mapsto f^{\prime} \circ \varphi$.
Note that on each $e \in E(G)$ with the model $(G, l)$ above, $\varphi^{\ast}(f^{\prime})$ has only multiples of $\operatorname{deg}_e(\varphi)$ as its slopes for any $f^{\prime} \in \operatorname{Rat}(\Gamma^{\prime})$.

\section{Main results}
	\label{section:main results}

In this section, we give proofs of Theorem \ref{thm1} and Proposition \ref{prop:morphismsemiring}.

First, we prove Theorem \ref{thm1}.
To do it, we will prepare multiple lemmas and an algorithm.
Algorithm \ref{alg:connection} gives a range of values for a proper connected subgraph of a metric graph to connect the chip firing move defined by it and another chip firing move  (see Example \ref{ex1}).

\begin{algorithm}[H]
\caption{}
\label{alg:connection}
\begin{algorithmic}[1]
\renewcommand{\algorithmicrequire}{\textbf{Input:}}
\renewcommand{\algorithmicensure}{\textbf{Output:}}
\REQUIRE 
$\Gamma :$ a metric graph

$E(G_{\circ}) = \{ e_1, \ldots, e_n \}:$ a labeling of edges of the canonical model for $\Gamma$

$S:$ a proper connected subgraph of $\Gamma$

\ENSURE $l_S$

\STATE $i \leftarrow 1$

\WHILE{$i \le n$}

\IF{$e_i \cap S = \varnothing$}

\STATE $l_i \leftarrow (\text{the diameter of }\Gamma), i \leftarrow i + 1$

\ELSE[$S \supset e_i$]

\STATE $l_i \leftarrow (\text{the diameter of }\Gamma), i \leftarrow i + 1$

\ELSE[$S \supset \partial e_i$]

\STATE $l_i \leftarrow (\text{the length of }\overline{e_i \setminus S}) / 2, i \leftarrow i +1$

\ELSE[$S \subset e_i^{\circ}$]

\STATE $l_i \leftarrow \operatorname{min} \{ \operatorname{dist}( S, x) \,|\, x \text{ is one of the endpoints of } e_i \}, i \leftarrow i + 1$

\ELSE

\STATE $l_i \leftarrow (\text{the length of }\overline{e_i \setminus S}), i \leftarrow i + 1$

\ENDIF

\ENDWHILE

\STATE $l_S \leftarrow \operatorname{min} \{ l_1, \ldots, l_n \}$

\RETURN $l_S$
\end{algorithmic}
\end{algorithm}

In Algorithm \ref{alg:connection}, $\overline{e_i \setminus S}$ denotes the closure of $e_i \setminus S$, and if $S$ consists of only one point $x$, then we write $l_x$ instead of $l_{\{ x \}}$.

\begin{rem}
	\label{rem:connection}
\upshape{
Let $\Gamma$ be a metric graph and $S_1$ a proper connceted subgraph of $\Gamma$.
Let $l \le l_{S_1}$ and $S_2 := \{ x \in \Gamma \,|\, \operatorname{dist}(S_1, x) \le l \}$.
With $a := \operatorname{min}\{ k \in \boldsymbol{Z}_{>0} \,|\, l/k \le l_{S_2} \}$, $m := \operatorname{min}\{ k \in \boldsymbol{Z}_{>0} \,|\, l_{S_2} / k \le l / a \}$ and any $l^{\prime} > 0$, by the definition of chip firing moves, we have
\begin{eqnarray*}
\operatorname{CF} \left( S_2, \frac{l}{a} \right) &=& \operatorname{CF} \left( S_1, \frac{l}{a} \right)\\
&& \odot \bigodot_{k = 1}^a \bigodot_{\substack{x^{\prime} \in \Gamma: \\ \operatorname{dist}(S_1, x^{\prime}) = \frac{kl}{a}}} \left\{ \operatorname{CF} \left( \{ x^{\prime} \}, \frac{l}{a} \right) \odot \frac{l}{a} \right\},
\end{eqnarray*}
\begin{eqnarray*}
&& \operatorname{CF} \left( \left\{ x \in \Gamma \,|\, \operatorname{dist}(S_2, x) \le \frac{kl_{S_2}}{m} \right\}, \frac{l_{S_2}}{m} \right)\\
&=& \operatorname{CF} \left( \left\{ x \in \Gamma \,|\, \operatorname{dist}(S_2, x) \le \frac{(k - 1)l_{S_2}}{m} \right\}, \frac{l_{S_2}}{m} \right)\\
&& \odot \bigodot_{\substack{x^{\prime} \in \Gamma : \\ \operatorname{dist}(S_2, x^{\prime}) = \frac{kl_{S_2}}{m}}} \left\{ \operatorname{CF} \left( \{x^{\prime} \}, \frac{l_{S_2}}{m} \right) \odot \frac{l_{S_2}}{m} \right\},
\end{eqnarray*}
\begin{eqnarray*}
\operatorname{CF}(S_2, l_{S_2}) &=& \left\{ \operatorname{CF} \left(S_2, \frac{l}{a} \right) \oplus \left(- \frac{l_{S_2}}{m} \right) \right\}\\ 
&& \odot \bigodot_{k = 1}^{m - 1} \operatorname{CF} \left( \left\{ x \in \Gamma \,|\, \operatorname{dist}(S_2, x) \le \frac{kl_{S_2}}{m} \right\}, \frac{l_{S_2}}{m} \right),
\end{eqnarray*}
and 
\[
\operatorname{CF}(S_1, l + l^{\prime}) = \operatorname{CF}(S_1, l ) \odot \operatorname{CF}(S_2, l^{\prime}).
\]}
\end{rem}

Let $\Gamma$ be a metric graph.
Let $R$ be as in Section \ref{introduction}.
Let $(G_{\circ}, l_{\circ})$ be the canonical model for $\Gamma$.

\begin{lemma}
	\label{lem}
Let $e$ be an edge of $G_{\circ}$.
Let $x$ be in $e^{\circ}$.
Then, $\operatorname{CF}(\{ x \}, l_x) \in R$.
\end{lemma}

\begin{proof}
If $x$ is the midpoint of $e$, then $\operatorname{CF}(\{ x \}, l_x) = f_e \in R$.
Suppose that $x$ is not the midpoint of $e$.
Assume that $0 < l_x \le \frac{l_{\circ}(e)}{4}$ and $g_e(x) = -\frac{l_{\circ}(x)}{4}$.
Then
\begin{eqnarray*}
\operatorname{CF}(\{ x \}, l_x) &=& \left\{ \left( \frac{l_{\circ}(e)}{2} - l_x \right) \odot f_e \oplus \left(l_x - \frac{l_{\circ}(e)}{2} \right) \odot f_e^{\odot (-1)} \right\}^{\odot (-1)}\\
&& \odot \left( -\frac{l_{\circ}(e)}{4} \right) \odot g_e^{\odot (-1)} \oplus (-l_x) \in R.
\end{eqnarray*}

Similarly, if $0 < l_x \le \frac{l_{\circ}(e)}{4}$ and $h_e(x) = - \frac{l_{\circ}(e)}{4}$, then
\begin{eqnarray*}
\operatorname{CF}(\{ x \}, l_x) &=& \left\{ \left( \frac{l_{\circ}(e)}{2} - l_x \right) \odot f_e \oplus \left( l_x - \frac{l_{\circ}(e)}{2} \right) \odot f_e^{\odot (-1)} \right\}^{\odot (-1)}\\
&& \odot \left(- \frac{l_{\circ}(e)}{4} \right) \odot h_e^{\odot (-1)} \oplus (-l_x) \in R.
\end{eqnarray*}

When $\frac{l_{\circ}(e)}{4} < l_x \le \frac{l_{\circ}(e)}{3}$ and $g_e(x) = - \frac{l_{\circ}(e)}{4}$, we have
\begin{eqnarray*}
\operatorname{CF}(\{ x \}, l_x) &=& \left\{ \left( \frac{l_{\circ}(e)}{2} - l_x \right) \odot f_e \oplus \left( l_x - \frac{l_{\circ}(e)}{2} \right) \odot f_e^{\odot (-1)} \right\}^{\odot (-1)}\\
&& \odot \left\{ \left(- \frac{l_{\circ}(e)}{4} \right) \odot g_e^{\odot (-1)} \right\}^{\odot 2} \oplus (-l_x) \in R.
\end{eqnarray*}

Similarly, if $\frac{l_{\circ}(e)}{4} < l_x \le \frac{l_{\circ}(e)}{3}$ and $h_e(x) = - \frac{l_{\circ}(e)}{4}$, then
\begin{eqnarray*}
\operatorname{CF}(\{ x \}, l_x) &=& \left\{ \left( \frac{l_{\circ}(e)}{2} - l_x \right) \odot f_e \oplus \left( l_x - \frac{l_{\circ}(e)}{2} \right) \odot f_e^{\odot (-1)} \right\}^{\odot (-1)}\\
&& \odot \left\{ \left( - \frac{l_{\circ}(e)}{4} \right) \odot h_e^{\odot (-1)} \right\}^{\odot 2} \oplus (-l_x) \in R.
\end{eqnarray*}

When $\frac{l_{\circ}(e)}{3} < l_x < \frac{l_{\circ}(e)}{2}$ and $g_e(x) = - \frac{l_{\circ}(e)}{4}$, we have
\begin{eqnarray*}
\operatorname{CF}(\{ x \}, l_{\circ}(e) - 2l_x) &=& \left\{ \left( \frac{l_{\circ}(e)}{2} - l_x \right) \odot f_e \oplus \left( l_x - \frac{l_{\circ}(e)}{2} \right) \odot f_e^{\odot (-1)} \right\}^{\odot (-1)}\\
&& \odot \left\{ \left(- \frac{l_{\circ}(e)}{4} \right) \odot g_e^{\odot (-1)} \right\}^{\odot 2} \oplus (2l_x - l_{\circ}(e)) \in R.
\end{eqnarray*}

Similarly, if $\frac{l_{\circ}(e)}{3} < l_x < \frac{l_{\circ}(e)}{2}$ and $h_e(x) = - \frac{l_{\circ}(e)}{4}$, then
\begin{eqnarray*}
\operatorname{CF}(\{ x \}, l_{\circ}(e) - 2 l_x) &=& \left\{ \left( \frac{l_{\circ}(e)}{2} - l_x \right) \odot f_e \oplus \left( l_x - \frac{l_{\circ}(e)}{2} \right) \odot f_e^{\odot (-1)} \right\}^{\odot (-1)}\\
&& \odot \left\{ \left( - \frac{l_{\circ}(e)}{4} \right) \odot h_e^{\odot (-1)} \right\}^{\odot 2} \oplus (2l_x - l_{\circ}(e)) \in R.
\end{eqnarray*}

Let $x$ be in the fifth case.
Since
\begin{eqnarray*}
&&\operatorname{CF} \left( \left\{ x_1 \in \Gamma \,|\, \operatorname{dist}(x, x_1) \le \frac{l_{\circ}(e)}{2} - l_x \right\}, \frac{l_{\circ}(e)}{2} - l_x \right)\\
&=& \operatorname{CF} \left( \{ x \}, \frac{l_{\circ}(e)}{2} - l_x \right) \odot \bigodot_{\substack{x_1 \in e:\\ \operatorname{dist}(x, x_1) = \frac{l_{\circ}(e)}{2} - l_x }} \left\{ \operatorname{CF} \left( \{ x_1 \}, \frac{l_{\circ}(e)}{2} - l_x \right) \odot \left( \frac{l_{\circ}(e)}{2} - l_x \right) \right\} \in R,
\end{eqnarray*}
with inputs $l = \frac{l_{\circ}(e)}{2} - l_x$, $S_1 = \{ x_1 \in \Gamma \,|\, \operatorname{dist}(x, x_1) \le l_{\circ}(e) / 2 - l_x \}$ in Remark \ref{rem:connection}, we have $\operatorname{CF}(\{ x \}, l_x) \in R$.

When $x$ is in the sixth case, by the same argument, we have $\operatorname{CF}(\{ x \}, l_x) \in R$.
\end{proof}

Note that $l_x$ coincides with $\operatorname{min}(\operatorname{dist}(x, v), \operatorname{dist}(x, w))$ in the setting of Lemma \ref{lem}.

By Remark \ref{rem:connection} and Lemma \ref{lem}, we prove the following three lemmas.
Let $d$ be the diameter of $\Gamma$, i.e., $d = \operatorname{sup}\{ \operatorname{dist}(x, y) \,|\, x, y \in \Gamma \} = \operatorname{max}\{ \operatorname{dist}(x, y) \,|\, x, y \in \Gamma \}$.

\begin{lemma}
	\label{lem1}
For any $x \in \Gamma$ and any positive real number $l$, the chip firing move ${\rm CF}(\{ x \}, l)$ is in $R$.
\end{lemma}

\begin{proof}
For any $x \in \Gamma$ and $l > 0$, by the definition of chip firing moves, we have $\operatorname{CF}(\{ x \}, l) = \operatorname{CF}(\{ x \}, d) \oplus (-l)$.
Hence it is sufficient to check that $\operatorname{CF}(\{ x \}, d) \in R$.
If $x \in V(G_{\circ})$, then $\operatorname{CF}(\{ x \}, d) \in R$.

Suppose that there exists an edge $e \in E(G_{\circ})$ such that $x \in e^{\circ}$.
Considering Remark \ref{rem:connection} with $l = l_x, S_1 = \{x \}$, by Lemma \ref{lem}, we have 
\[
\operatorname{CF}(S_2, l_{S_2}) \in R
\]
and
\[
\operatorname{CF}(S_1, l + l_{S_2}) = \operatorname{CF}(S_1, l) \odot \operatorname{CF}(S_2, l_{S_2}) \in R.
\]
Since $S_2$ contains a lot of whole edges of $G_{\circ}$ more than $S_1$ and the set of edges of $G_{\circ}$ is finite, by repeating inputs of $l = l_{S_2}$, $S_1 = S_2$ in Remark \ref{rem:connection}, we have $\operatorname{CF}(\{ x \}, d) \in R$.
\end{proof}

\begin{lemma}
	\label{lem2}
For any proper connected subgraph $\Gamma_1$ and any positive real number $l$, the chip firing move $\operatorname{CF}(\Gamma_1, l)$ is in $R$.
\end{lemma}

\begin{proof}
By Lemma \ref{lem1}, if $\Gamma_1$ consists of only one point, then we have the conclusion.
Assume that $\Gamma_1$ does not consist of only one point.

Suppose that $\Gamma_1$ contains no whole edges of $G_{\circ}$ and that there exists an edge $e \in E(G_{\circ})$ containing $\Gamma_1$.
Let $x_1$ and $x_2$ be the endpoints of $\Gamma_1$.
Let $x$ be the midpoint of $\Gamma_1$.
By Lemma \ref{lem1}, for any positive real number $l$, we have
\[
\operatorname{CF}(\Gamma_1, l) = \left[ \left\{ \operatorname{CF}(\{ x \}, l + \operatorname{dist}(x_1, x)) \odot \operatorname{dist}(x_1, x) \right\}^{\odot (-1)} \oplus 0 \right]^{\odot (-1)} \in R.
\]
Note that $\operatorname{CF}(\Gamma_1, l)$ is also obtained as follows with a sufficiently large $b \in \boldsymbol{Z}_{>0}$:
\[
\operatorname{CF} \left(\Gamma_1, \frac{l_x}{b} \right) = \operatorname{CF} \left(\{ x \}, \frac{l_x}{b} \right) \odot \bigodot_{k = 1}^{b} \bigodot_{\substack{x^{\prime} \in \Gamma:\\ \operatorname{dist}(x, x^{\prime}) = \frac{kl_x}{b}}} \left\{ \operatorname{CF} \left(\{ x^{\prime} \}, \frac{l_x}{b} \odot \frac{l_x}{b} \right) \right\},
\]
and inputs $l = l_x / b$ and $S_1 = \Gamma_1$, and repeating this process with inputs $l = l_{S_2}$, $S_1 = S_2$ in Remark \ref{rem:connection}.

Suppose $\Gamma_1$ contains $p$ edges.
Let $\partial \Gamma_1 \cup (V(G_{\circ}) \cap \Gamma_1) = \{ x_1, \ldots, x_q \}$.
We may assume that $x_1, \ldots, x_q$ are distinct.
Let $\Gamma_{11}, \ldots, \Gamma_{1s}$ be connected components of $\Gamma_1 \setminus \{ x_1, \ldots, x_q \}$.
For a sufficiently small positive real number $\varepsilon$, let $\Gamma_{1i}^{\prime}$ be the connected subgraph $\{ x \in \Gamma_{1i} \,|\, \text{for any }j, \operatorname{dist}(x, x_j) \ge \varepsilon \}$ of $\Gamma$.
Then, we have
\begin{eqnarray*}
\operatorname{CF}(\Gamma_1, \varepsilon) = 
 \left\{ \bigoplus^q_{k=1} \operatorname{CF}(\{ x_k \}, \varepsilon) \right\}
\odot \bigodot_{k=1}^s \left( \varepsilon \odot \operatorname{CF}(\Gamma_{1k}^{\prime}, \varepsilon) \right).
\end{eqnarray*}
The last divisor is in the first case, and thus it is in $R$.
By inputting $l = \varepsilon$, $S_1 = \Gamma_1$ and by repeating inputs $l = l_{S_2}$, $S_1 = S_2$ in Remark \ref{rem:connection}, we have $\operatorname{CF}(\Gamma_1, d) \in R$.
From this, for any $l > 0$, we have $\operatorname{CF}(\Gamma_1, l) = \operatorname{CF}(\Gamma_1, d) \oplus (-l) \in R$.
\end{proof}

\begin{lemma}
	\label{lem3}
For any proper subgraph $\Gamma_1$ and any positive real number $l$, the chip firing move $\operatorname{CF}(\Gamma_1, l)$ is in $R$.
\end{lemma}

\begin{proof}
Let $\Gamma_1$ be a proper subgraph of $\Gamma$.
Let $s$ be the number of connected components of $\Gamma_1$.
If $s = 1$, then the conclusion follows Lemma \ref{lem2}.
Assume $s \ge 2$.
Let $\Gamma_1^{\prime}, \ldots, \Gamma_s^{\prime}$ be all the distinct connected components of $\Gamma_1$.
For $l^{\prime} > 0$, let $\Gamma^{\prime}_k (l^{\prime}) := \{ x \in \Gamma \,|\, \operatorname{dist}(\Gamma^{\prime}_k, x) \le l^{\prime} \}$.
If $l^{\prime}$ is sufficiently small, then the intersection of $\Gamma^{\prime}_1(l^{\prime}), \ldots, \Gamma^{\prime}_s(l^{\prime})$ is empty.
Let $l^{\prime}_1$ be the minimum value of $l^{\prime}$ such that this intersection is nonempty.
By induction on $s$, $\operatorname{CF} \left( \bigcup_{k = 1}^{s}{\Gamma_k^{\prime}(l^{\prime}_1)}, d \right) \in R$.
On the other hand,
\[
\operatorname{CF}(\Gamma_1, l^{\prime}_1) = \bigoplus_{k = 1}^s \operatorname{CF}(\Gamma_k^{\prime}, l_1^{\prime}) \in R.
\]
Hence
\[
\operatorname{CF}(\Gamma_1, d) = \operatorname{CF}(\Gamma_1, l^{\prime}_1) \odot \operatorname{CF} \left( \bigcup_{k = 1}^{s}{\Gamma_k^{\prime}(l^{\prime}_1)}, d \right) \in R.
\]
In conclusion, for any $l > 0$, we have
\[
\operatorname{CF}(\Gamma_1, l) = \operatorname{CF}(\Gamma_1, d) \oplus (-l) \in R. \qedhere
\]
\end{proof}

From Lemmas \ref{lem:chipfiring}, \ref{lem1}, \ref{lem2}, \ref{lem3}, we have the following proposition.

\begin{prop}
	\label{prop1}
Let $\Gamma$ be a metric graph.
Then, $\operatorname{Rat}(\Gamma)$ coincides with $R$.
In particular, it is finitely generated as a semifield ovar $\boldsymbol{T}$.
\end{prop}

Let us show Lemma \ref{lem:troptometric}:

\begin{proof}[Proof of Lemma \ref{lem:troptometric}]
There exists a natural inclusion $\iota : \Gamma^{\prime} \hookrightarrow \Gamma$ (cf. \cite{ABBR2}).
With this inclusion $\iota$, we have a natural inclusion $\kappa : \operatorname{Rat}(\Gamma^{\prime}) \hookrightarrow \operatorname{Rat}(\Gamma)$, i.e., for any $f^{\prime} \in \operatorname{Rat}(\Gamma^{\prime})$ and $x^{\prime} \in \Gamma^{\prime}$, $\kappa(f^{\prime})(\iota(x^{\prime})) = f^{\prime}(x^{\prime})$ and $\kappa(f^{\prime})$ is extended to be constant on each connected component of $\Gamma \setminus \iota(\Gamma^{\prime})$.
Let $\{ f_1^{\prime}, \ldots, f_n^{\prime} \}$ be a finite generating set of $\operatorname{Rat}(\Gamma^{\prime})$.
Let $L_1, \ldots, L_m$ be all the connected components of $\Gamma \setminus \iota(\Gamma)$.
Then $\{ \kappa(f_1^{\prime}), \ldots, \kappa(f_n^{\prime}), \operatorname{CF}(\overline{\Gamma \setminus L_1}, \infty), \ldots, \operatorname{CF}(\overline{\Gamma \setminus L_m}, \infty) \}$ is a finite generating set of $\operatorname{Rat}(\Gamma)$.
In fact, for any $f \in \operatorname{Rat}(\Gamma) \setminus \{ -\infty \}$, since $f$ is a piecewice affine function with a finite number of pieces, it breaks each $L_i$ into a finite number of pieces $L_{i1}, \ldots, L_{is_i}$ on each which it has a constant slope.
We may assume that $\overline{L_{ij}} \cap \{ \iota(\Gamma^{\prime}) \cup \bigcup_{k = 1}^{j - 1}\overline{L_{ik}} \} \not= \varnothing$.
Let $x_{i1}$ be the unique point of $\overline{L_{i1}} \cap \iota(\Gamma^{\prime})$.
For any $j = 2, \ldots, s_i$, let $x_{ij}$ be the unique point of $\overline{L_{i j}} \cap \overline{L_{i, j-1}}$.
Let $x_{i, s_i + 1}$ be the point at infinity of $\overline{L_{is_i}}$.
Let $a_{ij}$ be the slope of $f$ on $L_{ij}$ in the direction from $x_{i j}$ to $x_{i, j + 1}$.
Since the restriction $f|_{\iota(\Gamma^{\prime})}$ can be regarded as a rational function on $\Gamma^{\prime}$, it is written as $g(\kappa(f_1^{\prime}), \ldots, \kappa(f_n^{\prime})) \odot h(\kappa(f_1^{\prime}), \ldots, \kappa(f_n^{\prime}))^{\odot (-1)}$ with polynomials $g, h \in \boldsymbol{T}[X_1, \ldots, X_n]$.
Let $b_{ij}$ be the value $\operatorname{CF}(\overline{\Gamma \setminus L_i}, \infty)(x_{ij})$.
Then, we have
\begin{eqnarray*}
f &=& g(\kappa(f_1^{\prime}), \ldots, \kappa(f_n^{\prime})) \odot h(\kappa(f_1^{\prime}), \ldots, \kappa(f_n^{\prime}))^{\odot (-1)}\\
&& \odot
\bigodot_{i = 1}^m \bigodot_{j = 1}^{s_i} \left[ (- b_{ij}) \odot \left\{ \left( \operatorname{CF}(\overline{\Gamma \setminus L_i}, \infty) \oplus b_{i, j +1} \right)^{\odot (-1)} \oplus (-b_{ij}) \right\}^{\odot (-1)} \right]^{\odot (- a_{ij})},
\end{eqnarray*}
which completes the proof.
Here, when $a_{ij} = 0$, then the last divisor means the zero function $0 \in \operatorname{Rat}(\Gamma)$.
\end{proof}

In conclusion, we have Theorem \ref{thm1}.

\begin{rem}
	\label{rem:isomorphism}
\upshape{
By the proof of Theorem \ref{thm1}, we have the following:
for a tropical curve $\Gamma$, all chip firing moves defined by one finite point, and of the form $\operatorname{CF}(\Gamma \setminus (y, x], \infty)$ with $x \in \Gamma_{\infty}$ and a finite point $y$ on the unique edge incident to $x$ generate $\operatorname{Rat}(\Gamma)$ as a semifield over $\boldsymbol{T}$.
This assertion is used in the proof of \cite[Corollary 3.9]{JuAe3}.
}
\end{rem}

Since the pull-back of the rational function semifield of a tropical curve by a finite harmonic morphism contains $\boldsymbol{T}$, the following corollary follows from Theorem \ref{thm1}:

\begin{cor}
	\label{cor1}
Let $\varphi : \Gamma \to \Gamma^{\prime}$ be a finite harmonic morphism between tropical curves.
Then, $\operatorname{Rat}(\Gamma)$ is finitely generated as a semifield over $\varphi^{\ast}(\operatorname{Rat}(\Gamma^{\prime}))$.
\end{cor}

By the proof of Lemma \ref{lem:troptometric}, we have the following corollary:

\begin{cor}
	\label{cor2}
Let $\Gamma$ be a tropical curve.
Let $(G_{\circ}, l_{\circ})$ be the canonical model for $\Gamma$ and $E_{\infty}$ the subset of $E(G_{\circ})$ cosisting of all edges of length $\infty$.
Then, there exists a generating set of $\operatorname{Rat}(\Gamma)$ consisting of at most $\# V(G_{\circ}) + 2(\# E(G_{\circ}) -\# E_{\infty})$ elements.
\end{cor}

\begin{proof}
$V(G_{\circ})$ contains $\# V(G_{\circ}) - \# E_{\infty}$ vertices which are finite points.
Thus $R$ for the metric graph obtained from $\Gamma$ by contracting all edges in $E_{\infty}$ is generated by $\# V(G_{\circ}) - \# E_{\infty} + 2(\# E(G_{\circ}) - \# E_{\infty})$ elements.
From the proof of Lemma \ref{lem:troptometric}, $\operatorname{Rat}(\Gamma)$ is generated by $\# V(G_{\circ}) - \# E_{\infty} + 2(\# E(G_{\circ}) - \# E_{\infty}) + \# E_{\infty} = \# V(G_{\circ}) + 2(\# E(G_{\circ}) -\# E_{\infty})$ elements.
\end{proof}

Second, we consider rational function semifields of trees.

\begin{lemma}
	\label{lem:tree}
Let $T$ be a tree.
Let $(G_{\circ}, l_{\circ})$ be the canonical model for $T$.
Let $V_1 \subset V(G_{\circ})$ denote the subset of all leaf ends.
If $\# V_1$ is even, then there exists a pairing of vertices in $V_1$ such that the union of unique paths connecting paired vertices covers $T$.
\end{lemma}

\begin{proof}
Since $\# V_1$ is even, there exists a pairing of vertices in $V_1$.
If it is not desired, then there exists an edge $e$ of $G_{\circ}$ which is not contained in the union of unique paths connecting paired vertices.
Since $T$ is a tree, there exist two vertices $v, w$ such that the unique path connecting them contains $e$.
By pairing again $v, w$ and the two other vertices $v^{\prime}, w^{\prime}$ originally paired with $v, w$ respectively, the number of covered edges increases.
In fact, the union of the path from $v$ to $w$ and the path from $v^{\prime}$ to $w^{\prime}$ contains $e$ and both the path from $v$ to $v^{\prime}$ and the path from $w$ to $w^{\prime}$.
Hence, by repeating this process, we have the conclusion.
\end{proof}

\begin{prop}
	\label{prop:tree}
Let $T$ be a tree.
Let $(G_{\circ}, l_{\circ})$ be the canonical model for $T$ and $V_1 \subset V(G_{\circ})$ the subset of all leaf ends.
Then there exists a generating set of $\operatorname{Rat}(T)$ consisting of at most $\frac{[\# V_1 + 1]}{2}$ elements, where $[x] = \max\{ n \in \boldsymbol{Z} \,|\, n \le x \}$.
\end{prop}

\begin{proof}
By Lemma \ref{lem:tree}, there exists a pairing of vertices in $V_1$ except at most one vertex $v_0$ such that the union of unique paths connecting paired vertices covers $T$ except at most one edge $e_0$ incident $v_0$.
Let $v, w$ be any paired vertices and $P$ the unique path connecting them.
Let $f$ be a rational function on $T$ which has slope one on $P$ in the direction from $v$ to $w$ and constant on other points.
Let $g$ be a rational function on $T$ which has slope one on $e_0$ and constant on other points.
Then such $f$ and $g$ generates $\operatorname{Rat}(T)$ as a semifield over $\boldsymbol{T}$.
In fact, for a tree $T^{\prime}$ which is a metric graph obtained from $T$ by contracting edges of length $\infty$, the restrictions of such $f$ and $g$ on $T^{\prime}$ generate $f_{e^{\prime}}, g_{e^{\prime}}, h_{e^{\prime}}$ and $\operatorname{CF}(\{ v \}, \infty)$ for each edge $e^{\prime}$ and each vertex $v$ of the underlying graph of the canonical model for $T^{\prime}$ and chip firing moves of the form of $\operatorname{CF}(\overline{\Gamma \setminus L_i}, \infty)$ in the proof of Lemma \ref{lem:troptometric}.
Hence we have the conclusion.
\end{proof}

Third, we show that except the singleton case, rational function semifields of tropical curves are not finitely generated as a \textit{$\boldsymbol{T}$-algebra} by the following two lemmas.

The following lemma holds by the definitions of two operators $\odot$, $\oplus$.

\begin{lemma}
	\label{lem:poly}
Let $\Gamma$ be a tropical curve.
For any rational functions $f, g \in \operatorname{Rat}(\Gamma) \setminus \{ -\infty \}$, $f \odot g$ and $f \oplus g$ may have as these poles only points that are poles of $f$ or $g$.
\end{lemma}

\begin{lemma}
	\label{lem:metricsemiring}
Let $\Gamma$ be a metric graph.
Then, $\operatorname{Rat}(\Gamma)$ is finitely generated as a $\boldsymbol{T}$-algebra if and only if $\Gamma$ is a singleton.
\end{lemma}

\begin{proof}
The if part is clear.
We shall show the only if part.
If $\operatorname{Rat}(\Gamma)$ is finitely generated as a $\boldsymbol{T}$-algebra, then by Lemma \ref{lem:poly}, $\Gamma$ must be a singleton.
\end{proof}

\begin{prop}
	\label{prop:tropsemiring}
Let $\Gamma$ be a tropical curve.
Then, $\operatorname{Rat}(\Gamma)$ is finitely generated as a $\boldsymbol{T}$-algebra if and only if $\Gamma$ is a singleton.
\end{prop}

\begin{proof}
By Lemma \ref{lem:metricsemiring}, it is enough to show that with \textit{any} metric graph $\Gamma^{\prime}$ obtained from $\Gamma$ by contracting edges of length $\infty$, if $\operatorname{Rat}(\Gamma)$ is finitely generated as a $\boldsymbol{T}$-algebra, then so is $\operatorname{Rat}(\Gamma^{\prime})$.
Assume that $\operatorname{Rat}(\Gamma)$ is finitely generated as a $\boldsymbol{T}$-algebra.
Let $\{ f_1, \ldots, f_n \}$ be a finite generating set of $\operatorname{Rat}(\Gamma)$.
Then the set of restrictions $\{ f_1|_{\Gamma^{\prime}}, \ldots, f_n|_{\Gamma^{\prime}} \}$ is a finite generating set of $\operatorname{Rat}(\Gamma^{\prime})$ with the natural inclusion $\Gamma^{\prime} \hookrightarrow \Gamma$.
In fact, the restriction map $\operatorname{Rat}(\Gamma) \to \operatorname{Rat}(\Gamma^{\prime})$ is surjective since the contraction $\Gamma \twoheadrightarrow \Gamma^{\prime}$ contracts only trees.
Hence, we have the assertion.
\end{proof}

Finally, we shall show Proposition \ref{prop:morphismsemiring}:

\begin{proof}[Proof of Proposition \ref{prop:morphismsemiring}]
Fix loopless models $(G, l)$, $(G^{\prime}, l^{\prime})$ for $\Gamma$, $\Gamma^{\prime}$, respectively, such that $\varphi(V(G)) = V(G^{\prime})$.
For any edge $e$ of $G$, if $e$ is not incident to a point at infinity, then let $F_e := \operatorname{CF}(\Gamma \setminus e^{\circ}, l(e)/2)$; otherwise, let $F_e := \operatorname{CF}(\overline{\Gamma \setminus e}, \infty)$.

Assume that $l(e) < \infty$.
Let $v$ be one of the vertices incident to $e$.
Let $G_{v, e}$ be a rational function on $\Gamma$ which
has slope one from $v$ to the midpoint of $e$;
has a sufficiently large positive integer to be its slope from $v$ to a point of each edge incident to $v$ other than $e$;
is the constant zero function on other points;
has $v$ as its unique point where attains the minimum value $-\frac{l(e)}{2}$.
Let $x \in e^{\circ}$.
Assume that $\operatorname{dist}(x, v) = l_x \le \frac{l(e)}{2}$.
If $\operatorname{dist}(x, v) = \frac{l(e)}{2}$, then we have
\begin{eqnarray*}
&&\operatorname{CF}(\{ x \}, l_x) \odot \{ -(\operatorname{deg}_e(\varphi) - 1)l_x\}\\
&=& \varphi^{\ast}(\operatorname{CF}(\{ \varphi(x) \}, l_{\varphi(x)})) \odot F_e^{\odot (\operatorname{deg}_e(\varphi) - 1)}
\odot \bigodot_{\substack{e_1 \in E(G) \setminus \{e \}:\\ e_1 \subset \varphi^{-1}(\varphi(e))}} F_{e_1}^{\odot \operatorname{deg}_{e_1}(\varphi)}.
\end{eqnarray*}
Suppose $\operatorname{dist}(x, v) \le \frac{l(e)}{4}$.
Let $f^{\prime}$ be a rational function on $\Gamma^{\prime}$ which coincides with $\operatorname{CF}(\{ \varphi(x) \}, l_{\varphi(x)})$ on $U^{\prime} := \{ x^{\prime} \in \Gamma^{\prime} \,|\, \operatorname{dist}(\varphi(x), x^{\prime}) \le l_{\varphi(x)}\}$;
has a sufficiently small negative slope $s$ from $U^{\prime}$ on the $\varepsilon$-neighborhood of $U^{\prime}$ with a sufficiently small positive real number $\varepsilon$ enough to be the restriction of $G_{v, e}$ on the inverse image of the $\varepsilon$-neighborhood of $U^{\prime}$ does not take zero;
is the constant $- l_{\varphi(x)} + s\varepsilon$ on other points.
Then, there exists a positive integer $b$ such that
\begin{eqnarray*}
&& \operatorname{CF}(\{ x \}, l_x) \odot \{ - (\operatorname{deg}_{e}(\varphi) - 1) l_x \} \\
&=& \varphi^{\ast}(f^{\prime}) \odot
\left\{ F_e \oplus G_{v, e} \odot \left( \frac{l(e)}{2} - 2l_x \right) \right\}^{\odot (\operatorname{deg}_e(\varphi) - 1)}\\
&& \odot \bigodot_{\substack{e_1 \in E(G) \setminus \{e \}:\\ e_1 \subset \varphi^{-1}(\varphi(e)), v \in e_1}} F_{e_1}^{\odot b} \odot \bigodot_{\substack{e_2 \in E(G) \setminus \{ e \}, v_2 \in V(G) \setminus \{ v \}:\\ e_2 \subset \varphi^{-1}(\varphi(e)), v_2 \in \varphi^{-1}(\varphi(v)),\\ v \not\in e_2, v_2 \in e_2}} G_{v_2, e_2}^{\odot b}\\
&& \oplus (- \operatorname{deg}_e(\varphi)l_x).
\end{eqnarray*}
Suppose $\frac{l(e)}{4} < \operatorname{dist}(x, v) < \frac{l(e)}{2}$.
Then, we have

\begin{eqnarray*}
&& \operatorname{CF} \left( \{ x \}, \frac{l(e)}{2} - l_x \right) \odot \{ - (\operatorname{deg}_{e}(\varphi) - 1) l_x \} \\
&=& \varphi^{\ast}(\operatorname{CF}(\{ \varphi(x) \}, l_{\varphi(x)}) \odot
\left\{ F_e \oplus G_{v, e} \odot \left( \frac{l(e)}{2} - 2l_x \right) \right\}^{\odot (\operatorname{deg}_e(\varphi) - 1)}\\
&& \odot \bigodot_{\substack{e_1 \in E(G) \setminus \{ e \}:\\ e_1 \subset \varphi^{-1}(\varphi(e))}} F_{e_1}^{\odot \operatorname{deg}_{e_1}(\varphi)} \oplus \left\{ -( \operatorname{deg}_e(\varphi) - 1) l_x - \left( \frac{l(e)}{2} - l_x \right) \right\}.
\end{eqnarray*}
Since
\begin{eqnarray*}
&&\operatorname{CF} \left( \left\{ x_1 \in \Gamma \,|\, \operatorname{dist}(x, x_1) \le \frac{l(e)}{4} - \frac{l_x}{2} \right\}, \frac{l(e)}{4} - \frac{l_x}{2} \right)\\
&=& \operatorname{CF} \left( \{ x \}, \frac{l(e)}{4} - \frac{l_x}{2} \right) \odot \bigodot_{\substack{x_1 \in e:\\ \operatorname{dist}(x, x_1) = \frac{l(e)}{4} - \frac{l_x}{2}}} \left\{ \operatorname{CF} \left( \{ x_1 \}, \frac{l(e)}{4} - \frac{l_x}{2} \right) \odot \left( \frac{l(e)}{4} - \frac{l_x}{2} \right) \right\},
\end{eqnarray*}
\begin{eqnarray*}
&&\operatorname{CF} \left( \left\{ x_1 \in \Gamma \,|\, \operatorname{dist}(x, x_1) \le \frac{l(e)}{2} - l_x \right\}, \frac{l(e)}{4} - \frac{l_x}{2} \right)\\
&=& \operatorname{CF} \left( \left\{ x_1 \in \Gamma \,|\, \operatorname{dist}(x, x_1) \le \frac{l(e)}{4} - \frac{l_x}{2} \right\}, \frac{l(e)}{4} - \frac{l_x}{2} \right)\\
&& \odot \bigodot_{\substack{x_1 \in e:\\ \operatorname{dist}(x, x_1) = \frac{l(e)}{2} - l_x}} \left\{ \operatorname{CF} \left( \{ x_1 \}, \frac{l(e)}{4} - \frac{l_x}{2} \right) \odot \left( \frac{l(e)}{4} - \frac{l_x}{2} \right) \right\},
\end{eqnarray*}
and $0 < 3 \left( \frac{l(e)}{4} - \frac{l_x}{2} \right) < \frac{l(e)}{2}$, with inputs $l = \frac{l(e)}{4} - \frac{l_x}{2}$, $S_1 = \{ x_1 \in \Gamma \,|\, \operatorname{dist}(x, x_1) \le l(e) / 2 - l_x \}$ in Remark \ref{rem:connection}, we can show that $\operatorname{CF}(\{ x \}, l_x)$ is generated by $F_e, G_{v, e}$ as a $\varphi^{\ast}(\operatorname{Rat}(\Gamma^{\prime}))$-algebra.

Assume that $l(e) = \infty$.
Identify $e = [0, \infty]$.
Let $x, y \in e$ be any distinct points such that $x < y$.
Let $g_{[x, y]}$ be the rational function on $\Gamma$ which has slope one in the direction from $x$ to $y$, is constant on any other points, and whose minimum value is zero.
Identify $\varphi(e) = [0, \infty]$.
Let $g^{\prime}_{[\varphi(x), \varphi(y)]}$ be the rational function on $\Gamma^{\prime}$ which has slope one in the direction from $\varphi(x)$ to $\varphi(y)$, is constant on any other points, and whose minimum value is zero.
Let $v$ be the endpoint $0$ of $e$.
Let $l_x$ (resp. $l_y$) be the length of $[v, x]$ (resp. $[v, y]$).
Then, we have
\[
g_{[x, y]} = g_{[v, y]} \odot (-l_x) \oplus 0
\]
and
\[
g_{[v, y]} = \varphi^{\ast} \left( g^{\prime}_{[\varphi(v), \varphi(y)]} \right) \odot \operatorname{CF}(\{ v \}, l_y)^{\odot (\operatorname{deg}_e(\varphi) - 1)} \odot \bigodot_{\substack{e_1 \in E(G) \setminus \{ e \}:\\ e_1 \subset \varphi^{-1}(\varphi(e))}} F_{e_1}^{\operatorname{deg}_{e_1}(\varphi)} \oplus 0.
\]
Let $w$ be the endpoint $\infty$ of $e$.
Then, we have
\[
g_{[x, y]}^{\odot (-1)} = g_{[x, w]}^{\odot (-1)} \oplus (-\operatorname{dist}(x, y))
\]
and
\[
g_{[x, w]}^{\odot (-1)} = F_e \odot g_{[v, x]}.
\]
Let $\Gamma_1$ be the metric graph obtained from $\Gamma$ by contracting all edges of $G$ of length $\infty$ to the finite endpoints.
We regard that $\Gamma_1$ is a subgraph of $\Gamma$.
By the proof of Theorem \ref{thm1}, for any rational function $f \in \operatorname{Rat}(\Gamma)$, the restriction $f|_{\Gamma_1}$ is generated by (the restrictions on $\Gamma_1$ of elements of) $\{ \operatorname{CF}(\{ x \}, l_x), \operatorname{CF}(\{ v \}, \infty) \,|\, x \in \Gamma_1 \setminus V(G), v \in V(G) \setminus \Gamma_{\infty} \}$ as a $\boldsymbol{T}$-algebra.
Let $g$ be the generated rational function such that $f|_{\Gamma_1} = g|_{\Gamma_1}$.
By (tropical) multiplying rational functions of the forms of $g_{[x, y]}^{\odot (\pm 1)}$ above to $g$, $f$ is made of $g$.
Hence $\{ F_e, G_{v_1, e_1}, \operatorname{CF}( \{v \}, \infty) \,|\, e \in E(G), v, v_1 \in V(G) \setminus \Gamma_{\infty}, e_1 \in E(G) \setminus E_{\infty}, v_1 \in e_1 \}$ generates $\operatorname{Rat}(\Gamma)$ as a $\varphi^{\ast}(\operatorname{Rat}(\Gamma^{\prime}))$-algebra, where $E_{\infty}$ denotes the subset of $E(G)$ consisting of all edges of length $\infty$.
\end{proof}

By Example \ref{ex2}, we know that $\operatorname{Rat}(\Gamma)$ may not be finitely generated as a $\varphi^{\ast}(\operatorname{Rat}(\Gamma^{\prime}))$-module.

\begin{ex}
	\label{ex2}
\upshape{
Let $\Gamma := [0,2]$ and $\Gamma^{\prime} := [0,1]$.
The map $\varphi : \Gamma \to \Gamma^{\prime}; x \mapsto x \text{ when $0 \le x \le 1$}; x \mapsto 2 - x \text{ when $1 < x \le 2$}$ is a finite harmonic morphism of degree two.
Assume that $\operatorname{Rat}(\Gamma)$ is finitely generated as a $\varphi^{\ast}(\operatorname{Rat}(\Gamma^{\prime}))$-module, i.e., there exist $f_1, \ldots, f_n \in \operatorname{Rat}(\Gamma) \setminus \{ -\infty \}$ such that $\operatorname{Rat}(\Gamma) = \bigoplus_{i = 1}^n \varphi^{\ast}(\operatorname{Rat}(\Gamma^{\prime})) \odot f_i$.
Let $x^{\prime}$ be a point of $[0, 1) \subset \Gamma^{\prime}$.
Then, for any pair of values $a, b \in \boldsymbol{R}_{>0}$, $\operatorname{Rat}(\Gamma)$ contains a rational function which has $a, b$ as values at each element $x_1, x_2$ of $\varphi^{-1}(x^{\prime})$.
For example, consider $(\operatorname{CF}(\{ x_1 \}, \varepsilon)^{\odot p} \odot a \oplus 0) \odot (\operatorname{CF}(\{ x_2 \}, \varepsilon)^{\odot q} \odot b \oplus 0) \in \operatorname{Rat}(\Gamma)$ with a small positive number $\varepsilon > 0$ and some $p, q \in \boldsymbol{Z}_{>0}$ such that $p \varepsilon > a$ and $q \varepsilon > b$.
On the other hand, for any $i$ and $f^{\prime} \in \operatorname{Rat}(\Gamma^{\prime})$ such that $(\varphi^{\ast}(f^{\prime}) \odot f_i)(x_1) = a$, since $a - f_i(x_1) = \varphi^{\ast}(f^{\prime})(x_1) = \varphi^{\ast}(f^{\prime})(x_2)$ hold, $(\varphi^{\ast}(f^{\prime}) \odot f_i)(x_2) = a - f_i(x_1) + f_i(x_2)$ holds.
Hence, if for any $j$, $b \not= a - f_j(x_1) + f_j(x_2)$ holds, then $\bigoplus_{i = 1}^n \varphi^{\ast}(\operatorname{Rat}(\Gamma^{\prime})) \odot f_i$ contains no rational functions which take $a, b$ at $x_1, x_2$ respectively.
It is a contradiction.
Therefore, $\operatorname{Rat}(\Gamma)$ is not finitely generated as a $\varphi^{\ast}(\operatorname{Rat}(\Gamma^{\prime}))$-module.
}
\end{ex}


\begin{thebibliography}{9}

\bibitem{ABBR2} Omid Amini, Matthew Baker, Erwan Brugall\'{e} and Joseph Rabinoff, \textit{Lifting harmonic morphisms II: tropical curves and metrized complexes}, Algebra Number Theory \textbf{9}(2)(2015), 267--315.

\bibitem{JuAe} Song JuAe, \textit{Generators of linear systems on tropical curves}, Hokkaido Mathematical journal \textbf{50}(1):55--76, 2021.

\bibitem{JuAe3} Song JuAe, {\it Semiring isomorphisms between rational function semifields of tropical curves induce isomorphisms between tropical curves}, arXiv:2110.08091.

\bibitem{Maclagan=Sturmfels} Diane Maclagan and Bernd Sturmfels, \textit{Introduction to tropical geometry}, Graduate Studies in Mathematics, Vol.~161. American Mathematical Soc., Providence, RI, 2015.
\end{thebibliography}
\end{document}